\documentclass[11pt,a4paper]{amsart}
\usepackage{amsmath, amsthm, amscd, amsfonts}

\newtheorem{theorem}{Theorem}[section]

\newtheorem{proposition}[theorem]{Proposition}
\newtheorem{corollary}[theorem]{Corollary}
\theoremstyle{definition}
\newtheorem{definition}[theorem]{Definition}
\newtheorem{example}[theorem]{Example}

\newtheorem{remark}[theorem]{Remark}
\numberwithin{equation}{section}

\newcommand{\de}{\delta}

\newcommand{\p}{\pi}

\newcommand{\I}{{\mathcal{I}}}

\newcommand{\lll}{{\mathcal{L}}}
\newcommand{\kk}{{\mathcal{K}}}
\newcommand{\ti}{{\mathcal{T}}}

\newcommand{\qq}{A/I}
\begin{document}

\title{ Relative amenability of Banach algebras}

\author{H. Ghahramani, W. Khodakarami, E. Feizi}

\address{Department of Mathematics, University of Kurdistan, P. O. Box 416, Sanandaj, Iran.}\email{h.ghahramani@uok.ac.ir; hoger.ghahramani@yahoo.com}
\address{Department of Mathematics, University of Kurdistan, P. O. Box 416, Sanandaj, Iran.}\email{w.khodakarami@sci.uok.ac.ir; wania.khodakarami@gmail.com}
\address{ Mathematics Department, Bu-Ali Sina University, 65174-4161, Hamadan, Iran.}\email{efeizi@basu.ac.ir}


\begin{abstract}
Let $A$ be a Banach algebra and $I$ be a closed ideal of $A$. We say that $A$ is amenable relative to $I$, if $A/I$ is an amenable Banach algebra. We study the relative amenability of Banach algebras and investigate the relative amenability of triangular Banach algebras and Banach algebras associated to locally compact groups. We generalize some of the previous known results by applying the concept of relative amenability of Banach algebras, especially, we present a generalization of Johnson's theorem in the concept of relative amenability. 
\end{abstract} \maketitle

\section{Introduction }

Let $A$ is a Banach algebra and $X$ be a Banach $A$-bimodule. A \textit{derivation} $D : A \to X$  is a bounded linear map satisfying
\[ D(ab) = D(a)\cdot b + a\cdot D(b) \quad  (a, b\in A).\]
A derivation $D$ is called \textit{inner derivation}, if there is $x\in X$ such that
\[ D(a)=a\cdot x-x\cdot a \quad (a\in A).\]
The space of all derivations from $A$ into $X$ is denoted by $Z^1(A,X)$, and $N^1(A,X)$ is the space of all inner derivations from $A$ into $X$. The \textit{first cohomology group} of $A$ with coefficient in $X$ is the quotient space
\[H^1(A,X)=Z^1(A,X)/N^1(A,X). \]
The dual space $X^*$ of Banach $A$-bimodule $X$, is a Banach $A$-bimodule with respect to the module operations defined by 
\[< a\cdot\lambda, x >=<\lambda, x\cdot a >, \quad < \lambda\cdot a, x >=<\lambda, a\cdot x >,\]
where $a\in A$, $ x\in  X$ and $\lambda \in X^*$; in this case $X^*$ is called the \textit{dual module} of $X$.
The Banach algebra $A$ is called \textit{amenable} if $H^1(A,X^*)=0$ for every Banach $A$-bimodule $X$ and \textit{weak amenable} if $H^1(A,A^*)=0$.
\par 
The notion of an amenable Banach algebra was introduced by Johnson in 1972, and it was based on the amenability of locally compact group $G$ \cite{j}. One of the basically consequences was that for a locally compact group $G$, the group algebra $L^1(G)$ is amenable if and only if the group $G$ is amenable. Since then amenability has become a main topic in Banach algebra theory and in harmonic analysis and the theory of amenability of Banach algebras has a fairly long history, see \cite{dals, runde} for a comprehensive survey of results of this type.
\par 
One of the interesting subjects is discussion about  the semigroup version of the Johnson's main theorem. However mentioned theorem is not true for topological semigroups. As an example, Duncan and Namioka in \cite{duncan}, showed that if $S$ is an amenable inverse semigroup, $l^1(S)$ isn't amenable, generally.
Also they proved that for a suitable closed ideal $I$ of $\ell^{1}(S)$, $\ell^{1}(S)/I\cong \ell^{1}(G_{S})$, which $G_{S}$ is group congruence of inverse semigroup $S$. So by \cite[Theorem 1]{duncan} and Johnson's theorem, we may observe that $\ell^{1}(S)/I$ is amenable if and only if $S$ is amenable. Hence, we can see a relationship between amenability of a quotient of Banach algebra $\ell^{1}(S)$ and amenability of inverse semigroup $S$.
\par 
Recall that a \textit{triangular Banach algebra} $Tri(A,M, B)$ is a Banach algebra of the form
\[ Tri(A,M,B)=\bigg\lbrace \begin{pmatrix}
  a & m \\
  0 & b
\end{pmatrix} \, : \, a\in A, m\in M, b\in B \bigg\rbrace, \]
under the usual matrix operations and $l^1$-norm, where $A$ and $B$ are Banach algebras and $M$ is a Banach $(A, B)$-bimodule. In \cite{medghalchi}, it is proven that if $M\neq 0$, then $Tri(A,M,B)$ is not amenable, and hence if $M =0$ then $Tri(A,M,B)$ is amenable if and only if $A$ and $B$ are amenable. Let $\kk=\begin{pmatrix}
  0 & M \\
  0 & 0
\end{pmatrix}$, then $\kk$ is a closed ideal in $Tri(A,M,B)$ and $Tri(A,M,B)/\kk\cong A\oplus B$ (isometric isomorphism). So $Tri(A,M,B)/\kk$ is amenable if and only if $A$ and $B$ are amenable.
\par 
Relative properties, specially relative amenability is one of
the most interesting concepts in group theory, and in the
specific case, so is co-amenability. Let $G$ be a locally compact
group and $N$ be a normal subgroup of $G$, $N$ is co-amenable in
$G$ if  the quotient group $G/N$ is amenable, see \cite{monod}. Popa in \cite{popa}, defined and studied a natural notion of amenability for a finite von Neumann algebras $M$ relative to a von Neumann subalgebra $N$ (or co-amenability of $N$ in $M$).
A longstanding open question of Connes \cite{conz} asks whether
any finite von Neumann algebra embeds into an ultraproduct of
finite-dimensional matrix algebras. Song in \cite{so}, proves
that von Neumann algebras which satisfying Popa's co-amenability
have Connes's embedding property. So, the relative amenability
concept of von Neumann algebra is a very useful and interesting
notion. The von Neumann algebra $M$ is amenable relative to von Neumann
subalgebra $N$ if there exists a norm one projection of
$<M,N>$ onto $M$. Monod and Popa in \cite [Corollary
7]{monod}, studied the relation between relative amenability  of
von Neumann algebra with the relative amenability of subgroups, as the following: Let $H$ be a normal subgroup of the discrete group $G$. Then  the group von Neumann algebras $\mathcal{L}(G)$ is amenable relative to $\mathcal{L}(H)$ if and only if $H$ is co-amenable in $G$. So, we can see a correspondence between the relative amenability of group von Neumann algebra and the co-amenability of underlie group. 
 \par 
Given these issues, the question arises as if $H$ is a closed normal subgroup of a locally compact group $G$, is there a suitable closed ideal of $L^1(G)$ such that $L^1(G)/I$ is amenable if and only if $G/H$ is an amenable group?
\par 
This notations leads us to consider the idea of studying the amenability of Banach algebra $A$, relative to a closed ideal $I$ of $A$. We have seen above that there are two types of Banach algebra that they aren't amenable generally, but an appropriate quotient of them is amenable. With respect to these cases, and with the motivation of relative amenability of groups and the question raised, we introduce the following concept in this paper which is a natural generalization of the concept of amenability of Banach algebras. We say that the Banach algebra $A$ is \textit{amenable relative to closed ideal} $I$ (briefly say $A$ is $I$-amenable), if $A/I$ is an amenable Banach algebra. In this article we study the properties of relative amenability of a Banach algebra and we also ask some questions about this concept. Then we investigate the relative amenability of triangular Banach algebras and Banach algebras associated to locally compact groups. We answer the questions in some cases and also generalize some of the previous known results by applying the concept of relative amenability of Banach algebras. Especially, we present a generalization of Johnson's theorem in the concept of relative amenability which is in fact the answer to the question raised above about the relative amenability of the group algebras $L^1(G)$. 
\par 
This paper is organized as follows. In Section 2, we introduced relative amenability, and verify several primarily properties of it. Also, we study some fundamental problems of relative amenability of Banach algebras and ask some questions about this concept. Section 2 is devoted to examining of hereditary properties for relative
amenability of Banach algebras. In Section 3, we study the relative amenability of triangular Banach algebras and Banach algebras associated to locally compact groups, and we get several results in this regard, especially, we generalize some of the previous known results in the concept of relative amenability. Moreover, we answer the questions raised is Section 2 in some cases.
\newpage
\section{Primarily properties }
\quad \quad
 In this section, we introduced relative amenability, and verify several of its primarily properties. Also, we study  some
 fundamental problems of relative amenability of Banach algebras.
\begin{definition}
Let $A$ be a Banach algebra and $I$ be a closed ideal in $A$. We call that $A$ is amenable relative to $I$ or briefly, we say that $A$ is $I$-amenable, if $\qq$ is an amenable Banach algebra.
\end{definition}
In the introduction section, we observed numbers of non-amenable Banach algebras which they were  $I$-amenable, for some suitable closed ideal $I$ of them. Also, it is clear that, if $A$ is an amenable Banach algebra then $0$ is a closed ideal in it, and $A$ is amenable relative to $0$. It  infer that this is a non-obvious definition, and we can consider the amenability relative to a closed ideal as a natural generalization of the amenability notion. Note that if $I$ is a closed ideal in a Banach algebra $A$ which is amenable, then by \cite[Corollary 2.3.2 and Theorem 2.3.10]{runde} $A$ is amenable if and only if $A$ is $I$-amenable. So in this case $I$-amenability coincides with the amenability.
\begin{remark}\label{1}
Let $A$ be a Banach algebra and $I$ be a closed ideal in $A$. If $X$ is a Banach $\qq$-bimodule, then it becomes to a Banach
$A$-bimodule by the following module operation:
 \[ a\cdot x=(a+I)\cdot x, \quad  x\cdot a=x\cdot (a+I)\quad (x\in X, a\in A).\]
 In this case, we denote the first group cohomology  of $A$ with coefficient in $X$ by  $H_I^1(A,X)$ (the achieved  module with the above module operation).
\end{remark}
Now, this notion gives some necessary and sufficient conditions to
consider amenability of $A$ relative to closed ideal $I$, in
terms of the first group cohomology $H_I^1(A,X^*)$, where $X$ is a
Banach $\qq$-bimodule.
\begin{theorem}\label{quo}
Let $A$ be a Banach algebra and $I$ be a closed ideal of it.
\begin{itemize}
\item[(i)] If $H_I^1(A,X^*)=0$ for every Banach $A/I$-bimodule $X$, then $A$ is $I$-amenable.
\item[(ii)] If $A$ is $I$-amenable and $\overline{I^2}=I$, then
$H_I^1(A,X^*)=0$, for every Banach $A/I$-bimodule $X$.
\end{itemize}
\end{theorem}
\begin{proof}
$(i)$ Let $X$ be a  Banach $A/I$-bimodule, and $\de:A/I\to X^*$ be a
continuous derivation. Consider $X$ as a Banach $A$-bimodule
(similarly to remark ~\ref{1}), and define $\widetilde{\de}:A\to
X^*$ by  $\widetilde{\de}=\de  \p$, where $\p:A\to A/I$ is the
quotient map. So  $\widetilde{\de}$ is a continuous derivation,
and by the hypothesis it is inner. Thus there is a $\phi$ in $X^*$
such that $\widetilde{\de}(a)=a\cdot\phi-\phi\cdot a$, for every
$a$ in $A$. Hence there is a $\phi$ in $X^*$ such that
$\de(a+I)=(a+I)\cdot\phi-\phi\cdot (a+I)$, for every $a+I$ in
$A/I$. It concludes that $\de$ is inner, and hence $H^1(A/I,X^*)=0$. Since $X$ is 
an arbitrary Banach $A/I$-bimodule, it follows that $A$ is $I$-amenable.
\par 
$(ii)$ Let $X$ be a  Banach $A/I$-bimodule, it turns to a Banach
$A$-bimodule by the mentioned module operation in Remark~\ref{1}. Assume that $\de:A\to X^*$  be a continuous derivation. If $a,\:b\in I$, then for every $x\in X$, we observe that
 \begin{equation}\label{d0}
 \begin{split}
<\de(ab), x>&=<a \de(b), x>+<\de(a) b, x>\\
&=<\de(b), x\cdot a>+<\de(a), b\cdot x>\\
&=<\de(b), x\cdot  (a+I)>+<\de(a), (b+I) \cdot x>\\
&=0.
\end{split}
\end{equation}
Therefore $\de(ab)=0$ for every $a,\:b\in I$. Assume that $c\in I$. By hypothesis, $c=\lim_{n\to\infty} \sum_{i=1}^{k_n}a_i b_i$, where $a_i , b_i \in I$ and $n, k_n \in \mathbb{N}$ . Since $\de$ is continuous,  $\de(c)=0$, by \eqref{d0}. Hence $\de(I)=0$. Now, we define $\overline{\de}:A/I\to
  X^*$, by $\overline{\de}(a+I)=\de(a)$. Since $\de(I)=0$, it follows that $\overline{\de}$ is well defined. Simply check that $\overline{\de}$ is a derivation. Also $\overline{\de}\p=\de$, where $\p$ is the quotient map. It follows from \cite[Proposition 5.2.2]{dals} that $\overline{\de}$ is continuous. By hypothesis $A/I$ is amenable, so there exists $\phi\in X^*$ such that for all $a\in A$, $\overline{\de}(a+I)=(a+I)\cdot\phi-\phi \cdot (a+I)$. Hence $\de(a)=a \cdot \phi-\phi\cdot a$ for all $a\in A$. Therefore $\de$ is inner and hence $H_I^1(A,X^*)=0$.
\end{proof}
If the closed ideal $I$ is weakly amenable as a Banach algebra or has a bounded approximate identity, then $\overline{I^2}=I$. The next example shows that we can't remove this condition in the statement $(ii)$ of the above theorem.
\begin{example}
Consider the triangular Banach algebra $T=
\begin{pmatrix}
  A & M \\
  0 & B
\end{pmatrix}$, where $A$ and $B$ are amenable unital Banach algebras and $M\neq 0$ is a unital Banach
$(A,B)$-bimodule. Let $I=\begin{pmatrix}
  0 & M \\
  0 & 0
\end{pmatrix}$. Then $I$ is a closed ideal in $T$ and $T/I \cong A\oplus B$. Hence $T$ is $I$-amenable. It is clear that $\overline{I^2}\neq I$, because $I^2=0$. Now, we show that there is a Banach $T/I$-bimodule $X$ such that $H_I^1(T,X^*)\neq 0$.
\par 
Since $T/I \cong A\oplus B$ and $A\oplus B$ is a closed subalgebra of $T$, it follows that $T^{*}$ is a Banach $T/I$-bimodule. So by the following module actions $T^{**}$ is a Banach $T/I$-bimodule (see \cite{f&m}):
\[ (a,b)\cdot \begin{pmatrix}
x^{**} & y^{**} \\
0 & z^{**}
\end{pmatrix}= \begin{pmatrix}
a\cdot x^{**} & a\cdot y^{**} \\
0 & b\cdot z^{**}
\end{pmatrix} \,\, \text{and} \,\,  \begin{pmatrix}
x^{**} & y^{**} \\
0 & z^{**}
\end{pmatrix}\cdot (a,b)= \begin{pmatrix}
x^{**}\cdot a & y^{**}\cdot b \\
0 & z^{**}\cdot b
\end{pmatrix},
\]
where $(a,b)\in A\oplus B$ and $\begin{pmatrix}
x^{**} & y^{**} \\
0 & z^{**}
\end{pmatrix} \in T^{**} =\begin{pmatrix}
A^{**} & M^{**}\\
0 & B^{**}
\end{pmatrix}$. Now from Remark~\ref{1}, by the following module actions $T^{**}$ becomes a Banach $T$-bimodule: 
\[ \begin{pmatrix}
a & m\\
0 & b
\end{pmatrix} \cdot\begin{pmatrix}
x^{**} & y^{**} \\
0 & z^{**}
\end{pmatrix}= \begin{pmatrix}
a\cdot x^{**} & a\cdot y^{**} \\
0 & b\cdot z^{**}
\end{pmatrix}\]
 and
 \[ \begin{pmatrix}
x^{**} & y^{**} \\
0 & z^{**}
\end{pmatrix}\cdot\begin{pmatrix}
a & m\\
0 & b
\end{pmatrix}= \begin{pmatrix}
x^{**}\cdot a & y^{**}\cdot b \\
0 & z^{**}\cdot b
\end{pmatrix},
\]
where $\begin{pmatrix}
a & m\\
0 & b
\end{pmatrix}\in T$ and $\begin{pmatrix}
x^{**} & y^{**} \\
0 & z^{**}
\end{pmatrix} \in T^{**} =\begin{pmatrix}
A^{**} & M^{**}\\
0 & B^{**}
\end{pmatrix}$. We will show that $H_I^1(T,T^{**})\neq 0$. Define the continuous linear map 
$\delta:T \to T^{**}$ by $\delta ( \begin{pmatrix}
  a & m \\
  0 & b
\end{pmatrix})=\begin{pmatrix}
  0 & \widehat{m} \\
  0 & 0
\end{pmatrix}$, where $\widehat{m}$ is the canonical embedding of $m$ in $M^{**}$ and $T^{**}$ is a Banach $T$-bimodule as mentioned above. It is easily checked that $\delta$ is a derivation. We prove that $\delta$ is not inner. Suppose that $\delta$ is inner,
so there is an element $\begin{pmatrix}
x^{**} & y^{**} \\
0 & z^{**}
\end{pmatrix}$ in $T^{**}$ such that for every $ \begin{pmatrix}
  a & m \\
  0 & b
\end{pmatrix}\in T$,
 \begin{equation*}
 \begin{split}
 \delta (\begin{pmatrix}
  a & m \\
  0 & b
\end{pmatrix})& = \begin{pmatrix}
  a & m \\
  0 & b
\end{pmatrix}\cdot\begin{pmatrix}
  x^{**} & y^{**}\\
  0 & z^{**}
\end{pmatrix} -\begin{pmatrix}
   x^{**} & y^{**}\\
  0 & z^{**}
\end{pmatrix}\cdot\begin{pmatrix}
 a & m \\
  0 & b
\end{pmatrix}\\ &=
\begin{pmatrix}
a\cdot x^{**}-x^{**}\cdot a & a\cdot y^{**}-y^{**} \cdot b\\
0 & b\cdot z^{**}-z^{**}\cdot b
\end{pmatrix}
\end{split}
\end{equation*}
 Hence $\widehat{m}=a\cdot y^{**}-y^{**}\cdot b$ for every $a,b\in A$ and $m\in M$. Let
$a=b=0$, then $\widehat{m}=0$ for all $m \in M$ and this infer that $M=0$. This is
a contradiction, because $M\neq 0$. Therefore $H_I^1(T,T^{**})\neq 0$.
\end{example}
Recall that an ideal $N$ of an algebra is \textit{nilpotent}, if $N^k=0$ for some non-negative integer $k$.
\par 
In the following result we provide a necessary condition for relative amenability.
\begin{theorem}
Let $A$ be Banach algebra and $I$ be a closed ideal of $A$ such that $A$ is $I$-amenable. Let $N$ be a non-zero closed nilpotent ideal in $A$ and let $R$ be a closed subspace complement of $N$ in $A$ such that $R+I$ and $N+I$ are closed. Then $(R+I)\cap N \neq 0$.
\end{theorem}
\begin{proof}
Suppose that $(R+I)\cap N = 0$. Clearly $R/I$ and $N/I$ are closed subspaces of $A/I$. Let $a\in A$. Since $A=R\oplus N$ as direct sum of closed subspaces, it follows that there are elements $r\in R$ and $n\in N$ such that $a+I= r+n+I$. Hence $A/I=R/I +N/I$. Let $r+I=n+I$, where $r\in R$ and $n\in N$. So there is an element $x\in I$ such that $r-n=x$. Thus $n=r-x\in (R+I)\cap N = 0$. So $R/I \cap N/I= I$. Therefore $A/I=R/I \oplus N/I$. Since $N/I$ is complemented in $A/I$ and $A/I$ is amenable, from \cite[Theorem 2.3.7]{runde} it follows that $N/I$ is amenable. On the other hand $N/I$ is a nilpotent ideal and hence $N/I=I$. So $0\neq N\subseteq I$, which is a contradiction. Hence $(R+I)\cap N \neq 0$.
\end{proof}
Let the Banach algebra $A$ is amenable relative to the closed ideal $I$. If $I$ is amenable, then $A$ is amenable, but the converse is not necessarily true. This is true for $C^*$-algebras. Connes and Haagerup proved that a $C^*$-algebra is amenable if and only if it is nuclear; so a $C^*$-algebras $A$ is amenable relative to a closed ideal $I$ if and only if $A/I$ is nuclear. By \cite[Theorem 6.5.3 and Page 216]{mur}, we have $A$ is nuclear if and only if $I$ and $A/I$ are both nuclear. Moreover, by \cite[Theorem 6.3.9]{mur} any finite-dimensional ideal in $A$ is nuclear and by \cite[Theorem 6.4.15]{mur} any commutative closed ideal in $A$ is nuclear. In view of these notes we have the following remark. 
\begin{remark}
Let $A$ be a $C^*$-algebra and $I$ be a closed ideal in $A$.
\begin{itemize}
\item[(i)] If $A$ is not amenable and $A$ is $I$-amenable, then $I$ is not finite-dimensional and is not commutative.
\item[(ii)] Let $A$ be $I$-amenable. Then $A$ is amenable if and only if $I$ is amenable. While this conclusion, in general, is not necessarily true for Banach algebras.
\end{itemize}
\end{remark}
We remind that every Banach algebra $A$ is amenable relative to trivial ideal $A$, and also every amenable Banach algebra is amenable relative to any closed ideal of it. The following example shows that it isn't necessary true that any Banach algebra $A$ has a non-trivial ideal $I\neq A$, such that $A$ is $I$-amenable.
\begin{example}\label{f2}
Let $\mathbb{F}_2$ be the free group on two generators, then the group $C^{*}$-algebra $C^*(\mathbb{F}_2)$ is a simple and non-amenable $C^*$-algebra (see \cite{bunce}). So
it hasn't a non-trivial ideal $I\neq C^*(\mathbb{F}_2)$ such that $C^*(\mathbb{F}_2)$ is $I$-amenable.
\end{example}
This notions motivated us to ask these questions: What kinds of non-amenable Banach algebras $A$ has a closed ideal $I\neq A$, such that $A$ is $I$-amenable? We answer this question for some cases, in the last section. Let $A$ be an $I$-amenable Banach algebra. For which types of closed ideals $J$ in $A$ in relation to the $I$, the Banach algebra $A$ is $J$-amenable? In the next results, we answer this question, in some cases.
\begin{theorem}\label{sub}
Let $A$ be a Banach algebra and $I$ be a closed ideal in $A$.
\begin{itemize}
\item[(i)] If $A$ is $I$-amenable and $J$ is an arbitrary closed ideal in $A$ such that $I\subseteq J$, then $A$ is $J$-amenable.
\item[(ii)] If $A$ is $I$-amenable and $J$ is an arbitrary closed ideal in $A$ such that $J\subseteq I$ and
$I$ is $J$-amenable, then $A$ is $J$-amenable.
\end{itemize}
\end{theorem}
\begin{proof}
$(i)$ Since $I$ and $J$ are closed ideals in $A$, it follows that $J / I$ is a closed ideal in Banach algebra $A / I$. By isomorphism theorems, the linear map $\theta : A / J \rightarrow (A / I)/(J / I)$ defined by $\theta (a+ J)=(a+I)+J/I$ ($a\in A$) is an algebraic isomorphism. For any $a\in A$, we have 
\begin{equation*}
\begin{split}
\Vert \theta (a+I) \Vert & = \Vert (a+I)+ J/I \Vert  \\&
= inf \lbrace \Vert (a+I)+x+I \Vert \, : \, x\in J\rbrace \\ &
=inf \lbrace \Vert a+x+I \Vert \, : \, x\in J\rbrace \\ &
\leq inf \lbrace \Vert a+x \Vert \, : \, x\in J\rbrace \\  &
= \Vert a+J \Vert. 
\end{split}
\end{equation*}
So $\theta$ is a bounded isomorphism. By inverse mapping theorem $A / J $ and $(A / I)/(J / I)$ are isomorph as Banach algebras. Since $A / I$ is amenable and $J / I$ is a closed ideal of it, from \cite[Corollary 2.3.2]{runde}, it follows that $A/J$ is amenable.
\par 
$(ii)$ By a similar method as before, it is proved that $A / I $ and $(A / J)/(I / J)$ are isomorph as Banach algebras. By hypothesis, it follows that $(A / J)/(I / J)$ and $I/J$ are amenable. Thus $A/ J$ is amenable, by \cite[theorem 2.3.10]{runde}.
\end{proof}
According to Theorem~\ref{sub}-$(i)$, It's seen that if $A$ is $I$-amenable and $J$ is a closed ideal in $A$, then $A$ is $\overline{I+J}$-amenable.
\begin{theorem}
Let $A$ be a unital Banach algebra and $I_{1},\cdots , I_{n}$ be proper distinct closed ideals such that $I_{i}+I_{j}=A$ when $i\neq j$. If $A$ is $I_{i}$-amenable for any $1\leq i \leq n$, then $A$ is $\bigcap _{i=1}^{n}I_{i}$-amenable.
\end{theorem}
\begin{proof}
Define $\theta : A / \bigcap _{i=1}^{n}I_{i} \rightarrow \bigoplus _{i=1}^{n}A/I_{i}$ by $\theta(a+  \bigcap _{i=1}^{n}I_{i} )= (a+I_{1},\cdots , a+I_{n} )$ ($a\in A$). By Chinese remainder theorem, $\theta$ is an algebraic isomorphism. Consider $\bigoplus _{i=1}^{n}A/I_{i}$ as $l_{\infty}$-direct sum of Banach algebras. For any $a\in A$, we have 
\begin{equation*}
\begin{split}
\Vert \theta(a+  \bigcap _{i=1}^{n}I_{i} )\Vert & = \Vert (a+I_{1},\cdots , a+I_{n} ) \Vert _{\infty}  \\&
= max \lbrace \Vert a_{1}+I \Vert , \cdots ,\Vert a_{n}+I \Vert \rbrace \\ &
\leq  \Vert a\Vert.
\end{split}
\end{equation*}
So $\theta$ is a bounded isomorphism. By inverse mapping theorem $A / \bigcap _{i=1}^{n}I_{i}$ and $\bigoplus _{i=1}^{n}A/I_{i}$ are isomorph as Banach algebras. Since each $A / I_{i}$ ($1\leq i \leq n$) is amenable, by \cite[page 55]{runde} it follows that $l_{\infty}$-direct sum $\bigoplus _{i=1}^{n}A/I_{i}$ is amenable. Hence $A / \bigcap _{i=1}^{n}I_{i}$ is amenable. 
\end{proof}
Now, we assume that $A$ is an $I$-amenable Banach algebra, a
natural question is how much we can minimize the ideal $I$? According to this question, we introduce the following definition. 
\begin{definition}
 Let $A$ be a Banach algebra. We say that the closed ideal $I$ of $A$ is a \textit{minimal co-amenable ideal}, if $A$ is  amenable relative to $I$ and for any closed ideal $J$ of $A$ with $J \subseteq I$ such that $A / J$ is amenable, then $I=J$.
\end{definition}
\begin{example}
In the following, we give some preliminary examples of above definition:
\begin{itemize}
\item[(i)] For every amenable Banach algebra, $0$ is a  minimal co-amenable
ideal.
\item[(ii)] For every non-amenable simple Banach algebra $A$ (see Examplesee ~\ref{f2}), the ideal $A$ is a minimal co-amenable ideal.
\end{itemize}
\end{example}
We ask two following natural questions: What kinds of Banach algebras has a minimal co-amenable ideal? If a Banach algebras has minimal co-amenable ideals, is there any identification of them? 
\par 
In the next theorem we give a sufficient condition for a closed ideal to be minimal co-amenable ideal.
\begin{theorem}\label{mini}
Let $A$ be a Banach algebra and $N$ be a complemented nilpotent closed ideal in $A$. If $A$ is $N$-amenable, then $N$ is a minimal co-amenable ideal.
\end{theorem}
\begin{proof}
Assume that $I$ is a closed ideal of $A$ such that $A/I$ is amenable and $I\subseteq N$. Let $R$ be a complemented closed subspace of $N$. So $A=R\oplus N$ as direct sum of closed subspaces. Suppose that $\lbrace r_k \rbrace _{k=1}^{\infty}$ and $\lbrace x_{k}\rbrace _{k=1}^{\infty}$ are sequences in $R$ and $I$, respectively such that $r_k + x_k \rightarrow r+n$ where $r\in R$ and $n\in N$. Since $A=R\oplus N$ and $I\subseteq N$, it follows that $\Vert r_k -r \Vert + \Vert x_k - n \Vert \rightarrow 0$. So $r+n \in R+I$ and hence $R+I$ is closed. Therefore, $R/I$ is a closed subspace of $A/I$. Moreover, it is clear that $N/I$ is a closed subspace of $A/I$.  Let $a\in A$. There are elements $r\in R$ and $n\in N$ such that $a+I= r+n+I$. Hence $A/I=R/I +N/I$.  Let $r+I=n+I$, where $r\in R$ and $n\in N$. So there is an element $x\in I$ such that $r-n=x$. Since $I\subseteq N$, it follows that $r=n+x\in R \cap N=0$. So $R/I \cap N/I= I$. Therefore $A/I=R/I \oplus N/I$. Since $N/I$ is complemented in $A/I$ and $A/I$ is amenable, it follows that $N/I$ is amenable. On the other hand $N/I$ is a nilpotent ideal. So $ N= I$ and hence $N$ is a minimal co-amenable ideal.
\end{proof}
In Section 4 we give examples which satisfy the assumptions of the preceding theorem. Moreover, we answer the above questions for some kinds of Banach algebras. 
\section{Hereditary properties }
In this section we study some hereditary properties
of relative amenability.
\begin{theorem}\label{epi}
Let $A$ be an $I$-amenable Banach algebra and $B$ be a Banach
algebra. If $\phi:A\to B$ is a continuous epimorphism, then $B$ is $\overline{\phi(I)}$-amenable.
\end{theorem}
\begin{proof}
Define $\widetilde{\phi}:A/I\to B/\overline{\phi(I)}$ by $\widetilde{\phi}(a+I)=\phi(a)+\overline{\phi(I)}$. It is easily checked that $\widetilde{\phi}$ is well-defined and it is a surjective homomorphism. Let $\pi_{A}: A\rightarrow A/I$ and $\pi_{B}:B\rightarrow B/\overline{\phi(I)}$ be quotient maps. So $\widetilde{\phi}\pi_{A}=\pi_{B}\phi$. Hence $\widetilde{\phi}\pi_{A}$ is continuous and from \cite[Proposition 5.2.2]{dals}-$(i)$, it follows that $\widetilde{\phi}$ is continuous. Since $A/I$ is amenable, by \cite[Proposition 2.3.1]{runde} we see that $B/\overline{\phi(I)}$ is amenable.
\end{proof}
\begin{remark}\label{ker}
Let $\phi:A\to B$ be a continuous epimorphism. It is clear that $Ker(\phi)$ is a closed ideal of $A$. By using a similar method as the proof of Theorem~\ref{epi} and inverse mapping theorem, we can see $A/Ker(\phi)\cong B$ as isomorphism of Banach algebras. So $A$ is $Ker(\phi)$-amenable if and only if $B$ is amenable. 
\end{remark}
\begin{theorem}\label{c0}
Let $\{A_{\lambda}\}_{\lambda\in\Lambda}$ be a collection of $\{I_{\lambda}\}_{\lambda\in\Lambda}$-amenable Banach algebras, where each $I_\lambda$ is a closed ideal in $A_\lambda$ ($\lambda\in\Lambda$). Then $\oplus_{\lambda}^{c_0}{A_{\lambda}}$ is ${\oplus_{\lambda}^{c_0} I_{\lambda}}$-amenable.
\end{theorem}
\begin{proof}
It is obvious that $\oplus_{\lambda}^{c_0} I_{\lambda}$ is a closed ideal in $\oplus_{\lambda}^{c_0}A_{\lambda}$. Define $\theta : \oplus_{\lambda}^{c_0}A_{\lambda} /  \oplus_{\lambda}^{c_0} I_{\lambda}\rightarrow \oplus_{\lambda}^{c_0}A_{\lambda} / I_{\lambda} $ by $\theta ((a_{\lambda})+\oplus_{\lambda}^{c_0} I_{\lambda})=(a_{\lambda}+I_{\lambda})$ ($(a_{\lambda})\in \oplus_{\lambda}^{c_0}A_{\lambda}$). $\theta$ is a well-defined and by isomorphism theorems, it is easily checked that $\theta$ is an algebraic isomorphism. The linear map $\theta \pi$ is continuous, where $\pi :\oplus_{\lambda}^{c_0}A_{\lambda} \rightarrow \oplus_{\lambda}^{c_0}A_{\lambda} /  \oplus_{\lambda}^{c_0} I_{\lambda}$ is the quotient map. From \cite[Proposition 5.2.2]{dals}-$(i)$, it follows that $\theta$ is continuous. By inverse mapping theorem $\oplus_{\lambda}^{c_0}A_{\lambda} /  \oplus_{\lambda}^{c_0} I_{\lambda}$ and $\oplus_{\lambda}^{c_0}A_{\lambda} / I_{\lambda}$ are isomorph as Banach algebras. Since each $A_{\lambda}/ I_{\lambda}$ are amenable, $\oplus_{\lambda}^{c_0}A_{\lambda} / I_{\lambda}$ is amenable again (see \cite[Corollary 2.3.19]{runde}). This concludes that $\oplus_{\lambda}^{c_0}A_{\lambda}$ is $\oplus_{\lambda}^{c_0} I_{\lambda}$-amenable.
\end{proof}
The following result is immediate from Theorem~\ref{c0}
\begin{corollary}
Let  $A_{1},\cdots , A_{n}$ be $I_1 , \cdots . I_n$-amenable Banach algebras, respectively. Then $A_{1}\oplus_{\infty}\cdots \oplus_{\infty} A_n$ is $I_{1}\oplus_{\infty}\cdots \oplus_{\infty} I_n$-amenable.
\end{corollary}
\begin{remark}
Let $A$ be a Banach algebra and $I$ be a closed ideal in $A$.
\begin{itemize}
\item[(i)] Let $A^\sharp$ be the unitization of $A$. Then $I$ is a closed ideal in $A^\sharp$ and it is easily checked that $(A/I)^{\sharp}\cong A^\sharp /I$. From  \cite[Corollary 2.3.11]{runde}, it follows that $A$ is $I$-amenable if and only if $A^\sharp$ is $I$-amenable.
\item[(ii)] Let $A$ be an $I$-amenable Banach algebra, and $J$ be a closed ideal in $A$ with $I\subseteq J$. If $J$ has a bounded approximate identity, then $J$ is $I$-amenable. Because $J / I$ is a closed ideal in $A/ I$ with a bounded approximate identity ($J$ has a bounded approximate identity) and from \cite[ Proposition 2.3.3]{runde}, $J/I$ is amenable.
\item[(iii)] Let $(A^{**},\square)$ be the second dual of the Banach algebra $A$ which is equipped with the first Arens product. In the case where $I$ is a closed ideal in $A$, we have the identification (\cite[Page 250]{dals})
\[((A/I)^{**},\square )\cong (A^{**},\square)/I^{**}. \]
So if $(A^{**},\square)$ is $I^{**}$-amenable, then $((A/I)^{**},\square )$ is amenable and hence by \cite{gor}, $A$ is $I$-amenable.
\end{itemize}
\end{remark}
\section{Relative amenability of special Banach algebras}
In this section, we give various types of relative amenable Banach algebras and we answer some of the questions in Section 2.
\subsection*{Triangular Banach algebras}
From this point up to the last subsection $A$ and $B$ are Banach algebras, $M$ is a Banach $(A,B)$-bimodule. Also $Tri(A,M,B)$ denotes the triangular Banach algebra $\begin{pmatrix}
  A & M \\
  0 & B
\end{pmatrix}$ as defined in the Introduction section with $l^1$-norm. Let $C\subseteq A$, $D\subseteq B$ and $E\subseteq M$, then $Tri(C,E,D)$ denotes the subset 
\[ Tri(C,E,D)=\bigg\lbrace \begin{pmatrix}
  x & y \\
  0 & z
\end{pmatrix} \, : \, x\in C, y\in E, z\in D \bigg\rbrace \]
of $Tri(A,M,B)$.
\begin{remark}\label{ideal}
Let $\ti =Tri(A,M,B)$, $I$ be a closed ideal of $A$ and $J$ be a closed ideal of $B$. Then it is easily checked that $\kk = Tri(I,M,J)$ is a closed ideal of $\ti$ and $\ti / \kk $ and $ A/I \oplus B/J$ are isometrically isomorphic, where $ A/I \oplus B/J$ is the $l^1$-direct sum of Banach algebras. So $\ti$ is $\kk$-amenable if and only if $A$ is $I$-amenable and $B$ is $J$-amenable. Especially, we have the followings:
\begin{itemize}
\item[(i)] $\ti$ is $Tri(0,M,B)$-amenable if and only if $A$ is amenable.
\item[(ii)] $\ti$ is $Tri(A,M,0)$-amenable if and only if $B$ is amenable.
\item[(iii)] $\ti$ is $Tri(0,M,0)$-amenable if and only if $A$ and $B$ are amenable.
\end{itemize}
\end{remark}
By Johnson's theorem \cite{j} (a locally compact group $G$ is amenable if and only if $L^{1}(G)$ is an amenable Banach algebra) and Remark~\ref{ideal}, the following example is immediate.
\begin{example}
Let $G$ be a locally compact group and $M$ be a Banach $L^1(G)$-bimodule. Let $\I$ be the any of the closed ideals $Tri(0, M, L^1(G))$, $Tri(L^1(G), M, 0)$ or $Tri(0, M, 0)$ of triangular Banach algebra $\ti =Tri(L^1(G), M, L^1(G))$. Then $\ti$ is $\I$-amenable if and only if $G$ is an amenable group.
\end{example}
In the next result we give a necessary condition for a closed ideal $Tri(I,M,J)$ to be minimal co-amenable ideal.
\begin{proposition}\label{mm}
Let $\ti =Tri(A,M,B)$, $I$ be a closed ideal of $A$ and $J$ be a closed ideal of $B$. If $\kk=Tri(I,M,J)$ is a  minimal co-amenable ideal of $\ti$, then $I$ and $J$ are minimal co-amenable ideals of $A$ and $B$, respectively. 
\end{proposition}
\begin{proof}
By Remark~\ref{ideal}, $A/I$ and $B/J$ are amenable. Let $I^\prime \subseteq I$ and $J^\prime \subseteq J$ be closed ideals such that $A/I^\prime$ and $B/J^\prime$ are amenable. By Remark~\ref{ideal}, $\kk^\prime =Tri(I^\prime,M,J^\prime)$ is a closed ideal of $\ti$ such that $\ti$ is $ \kk^\prime$-amenable. Also $\kk^\prime \subseteq \kk$ and $\kk$ is a minimal co-amenable ideal. So $\kk=\kk^\prime$ and hence $I=I^\prime$ and $J=J^\prime$. Therefore, $I$ and $J$ are minimal co-amenable ideals of $A$ and $B$, respectively. 
\end{proof}
The following proposition provides some examples which satisfy the assumptions of the Theorem~\ref{mini}. Indeed, we find a minimal co-amenable ideal in some triangular Banach algebras.
\begin{proposition}\label{mm2}
Let $A$ and $B$ be amenable Banach algebras. Then the closed ideal $Tri(0,M,0)$ is a minimal co-amenable ideal in the $Tri(A,M,B)$.
\end{proposition}
\begin{proof}
By Remark~\ref{ideal}, $Tri(A,M,B)$ is amenable relative to $Tri(0,M,0)$. It is clear that $Tri(0,M,0)$ is complemented in $Tri(A,M,B)$ (with complemented closed subspace $Tri(A,0,B)$) and $Tri(0,M,0)^2=0$. So by Theorem~\ref{mini}, $Tri(0,M,0)$ is a minimal co-amenable ideal.
\end{proof}
If $A$ and $B$ are unital Banach algebras with unities $1_A$ and $1_B$, respectively and $M$ is a unital Banach $(A,B)$-bimodule, then the triangular Banach algebra $Tri(A,M,B)$ is unital with the unity $\textbf{1}=\begin{pmatrix}
  1_A & 0 \\
  0 & 1_B
\end{pmatrix}$. In continue we characterize all of the closed ideals $\lll$ in a unital triangular Banach algebra $\ti=Tri(A,M,B)$ such that $\ti$ is $\lll$-amenable. Moreover, in this case we describe all of the minimal co-amenable ideals in $\ti$. 
\begin{theorem}\label{uni}
Let $\ti=Tri(A,M,B)$ be a unital triangular Banach algebra and $\lll$ be a closed ideal of it. Then
\begin{itemize}
\item[(i)] $\ti$ is $\lll$-amenable if and only if there are closed ideals $I$ and $J$ of $A$ and $B$, respectively such that $\lll=Tri(I,M,J)$, $A$ is $I$-amenable and $B$ is $J$-amenable;
\item[(ii)] $\lll$ is a minimal co-amenable ideal of $\ti$ if and only if there are closed ideals $I$ and $J$ of $A$ and $B$, respectively such that $\lll=Tri(I,M,J)$, $I$ and $J$ are minimal co-amenable ideals of $A$ and $B$, respectively.
\end{itemize}
\end{theorem}
\begin{proof}
$(i)$ Suppose that  $\ti$ is $\lll$-amenable. Let $\textbf{1}=\begin{pmatrix}
  1_A & 0 \\
  0 & 1_B
\end{pmatrix}$ be the unity of $\ti$, where $1_A$ is the unity of $A$ and $1_B$ is the unity of $B$. We consider the elements $\begin{pmatrix}
  a & 0 \\
  0 & 0
\end{pmatrix}$, $\begin{pmatrix}
  0 & m \\
  0 & 0
\end{pmatrix}$ and $\begin{pmatrix}
  0 & 0 \\
  0 & b
\end{pmatrix}$ equal to $a$, $b$ and $b$, respectively, where $a\in A$, $m\in M$ and $b\in B$.
\par 
Define 
\[ I=\lll \cap Tri(A,0,0), \quad N=\lll \cap Tri(0,M,0) \quad \text{and} \quad J=\lll \cap Tri(0,0,B).\]
It is easily checked that $I$ is a closed ideal of $A$, $J$ is a closed ideal of $B$ and $N$ is a closed $(A,B)$-sub-bimodule of $M$ such that $IM+MJ\subseteq N$. Moreover, $Tri(I,N,J)\subseteq \lll$. Let $\begin{pmatrix}
  x & y \\
  0 & z
\end{pmatrix}\in \lll$. Since $\lll$ is an ideal, it follows that 
\[\begin{pmatrix}
  x & 0 \\
  0 & 0
\end{pmatrix}=\begin{pmatrix}
  x & y \\
  0 & z
\end{pmatrix}\begin{pmatrix}
  1_A & 0 \\
  0 & 0
\end{pmatrix}\in I,\] 
\[\begin{pmatrix}
  0 & 0 \\
  0 & z
\end{pmatrix}=\begin{pmatrix}
  0 & 0 \\
  0 & 1_B
\end{pmatrix}\begin{pmatrix}
  x & y \\
  0 & x
\end{pmatrix}\in J \] 
and 
\[ \quad \quad \begin{pmatrix}
  0 & y \\
  0 & 0
\end{pmatrix}=\begin{pmatrix}
  1_A & 0 \\
  0 & 0
\end{pmatrix}\begin{pmatrix}
  x & y \\
  0 & x
\end{pmatrix}\begin{pmatrix}
  0 & 0 \\
  0 & 1_B
\end{pmatrix}\in N .\]
So $\begin{pmatrix}
  x & y \\
  0 & z
\end{pmatrix}\in Tri(I,N,J)$. Therefore $\lll=Tri(I,N,J)$. It can be seen directly that $\ti / \lll$ and $Tri(A/I, M/N,B/J)$ are isometrically isomorphic ($M/N$ becomes a Banach $(A/I, B/J)$-bimodule by well-defined operations). From hypothesis we have $Tri(A/I, M/N,B/J)$ is an amenable Banach algebra. Since $Tri(0,M/N,0)$ is a closed nilpotent ideal which is complemented in $Tri(A/I, M/N,B/J)$, from \cite[Theorem 2.3.7]{runde} it follows that $N=M$. So $\lll =Tri(I,M,J)$. By Remark~\ref{ideal}, $A/I$ and $B/J$ are amenable. 
\par
The converse is immediate from Remark~\ref{ideal}.
\par 
$(ii)$ Suppose that $\lll$ is a minimal co-amenable ideal of $\ti$. So $\ti$ is $\lll$-amenable and by $(i)$ we have $\lll = Tri(I,M,J)$ where $I$ and $J$ are closed ideals of $A$ and $B$, respectively and $A/I$ and $B/J$ are amenable. Now by Proposition~\ref{mm}, we obtain the desired result. 
\par 
Conversely, let $\lll=Tri(I,M,J)$, $I$ and $J$ are minimal co-amenable ideals of $A$ and $B$, respectively. By Remark~\ref{ideal}, $\ti / \lll$ is amenable. Suppose that $\lll^\prime \subseteq \lll$ be a closed ideal such that $\ti / \lll^\prime$ is amenable. From $(i)$ there are closed ideals $I^\prime$ and $J^\prime$ of $A$ and $B$, respectively such that $\lll^\prime=Tri(I^\prime,M,J^\prime)$, $A$ is $I^\prime$-amenable and $B$ is $J^\prime$-amenable, respectively. So $I^\prime \subseteq I$ and $J^\prime \subseteq J$. Since $I$ and $J$ are minimal co-amenable ideals, it follows that $I^\prime = I$ and $J^\prime = J$. So $\lll = \lll^\prime$ and hence $\lll$ is a minimal co-amenable ideal of $\ti$.
\end{proof}
This theorem provides the converse of Proposition~\ref{mm} in the case of unital triangular Banach algebras. Also this theorem shows that if $Tri(A,M,B)$ is a unital triangular Banach algebra and $A$ and $B$ are amenable Banach algebras, then $Tri(0,M,0)$ is unique minimal co-amenable ideal of $Tri(A,M,B)$. Moreover, by the above theorem we can obtain minimal co-amenable ideals of various unital triangular Banach algebras.
\begin{example}
Let $A$ be a unital amenable Banach algebra, $B$ be a unital simple and non-amenable Banach algebra and $M$ be a unital Banach $(A,B)$-bimodule. Then by Theorem~\ref{uni}, the closed ideal $Tri(0,M,B)$ is unique minimal co-amenable ideal of $Tri(A,M,B)$, since $0$ and $B$ are unique minimal co-amenable ideals of $A$ and $B$, respectively.
\end{example} 
In the end of this subsection we study the relative amenability of upper triangular Banach algebras. The \textit{upper triangular Banach algebra} $T_n(A)$ ($n\geq 1$) contains all $n\times n$ upper triangular matrices over $A$ with $l^1$-norm. The nilpotent closed ideal of all strictly upper triangular matrices in $T_n(A)$ is denoted by $NT_n(A)$ and the closed subalgebra of all diagonal matrices in $T_n(A)$ is denoted by $D_n(A)$. 
\begin{theorem}\label{tn}
Let $T_n(A)$ ($n\geq 1$) be the upper triangular Banach algebra over the Banach algebra $A$. Then 
\begin{itemize}
\item[(i)] $T_n(A)$ is amenable relative to $NT_n(A)$ if and only if $A$ is amenable;
\item[(ii)] if $A$ is amenable, then $NT_n(A)$ is a minimal co-amenable ideal of $T_n(A)$. 
\end{itemize}
\end{theorem}
\begin{proof}
$(i)$ We have $T_n(A)=D_n(A)\oplus NT_n(A)$ as direct sum of closed subspaces and $D_n(A)=A^n$. We can easily see that $T_n(A) /NT_n(A)$ and $D_n(A)$ are isometrically isomorphic. So $T_n(A) /NT_n(A)$ is amenable if and only if $A$ is amenable.
\par 
$(ii)$ If $A$ is amenable, by $(i)$, $T_n(A) /NT_n(A)$ is amenable. Since $NT_n(A)$ is a nilpotent complemented closed ideal in $T_n(A)$, from Theorem~\ref{mini}, it follows that $NT_n(A)$ is a minimal co-amenable ideal of $T_n(A)$.
\end{proof}
We have the following examples which satisfy the assumptions of the preceding theorem.
\begin{example}
The upper triangular Banach algebra $T_n(\mathbb{C})$ is $NT_n(\mathbb{C})$-amenable and $NT_n(\mathbb{C})$ is a minimal co-amenable ideal of $T_n(\mathbb{C})$.
\end{example}
\begin{example}
Let $G$ be a locally compact group. By Johnson's theorem and Theorem~\ref{tn}, we find that $T_n(L^1(G))$ is $NT_n(L^1(G))$-amenable if and only if $G$ is an amenable group. Moreover, if $G$ is an amenable group, then $NT_n(L^1(G))$ is a minimal co-amenable ideal of $T_n(L^1(G))$.
\end{example}
\subsection*{Banach algebras associated to locally compact groups}
\par 
Johnson's theorem \cite{j} states that a locally compact group $G$ is amenable if and only if $L^{1}(G)$ is an amenable Banach algebra. We present a generalization of this theorem in the concept of relative amenability which is in fact the answer to the question raised in Section 1 about the relative amenability of the group algebras $L^1(G)$. To do this, we introduce special ideals of $L^{1}(G)$ that are associated with normal subgroups of $G$. We provide the topics needed for this as follows from \cite[Chapter 3]{reit}.
\par
Let $H$ be a closed normal subgroup of $G$. Let $dx$, $d\xi$ and $d\dot{x}$ be Haar measures on $G$, $H$, and $G/H$ respectively, which are canonically related ($d\xi d\dot{x}=dx$) and consider the mapping $T_{H}:C_{c}(G)\rightarrow C_{c}(G/H)$ defined by $T_{H}f(\dot{x})=\int_{H}f(x\xi)d\xi$ ($f\in C_{c}(G), \dot{x}\in G/H $), where $x$ is any element of $G$ such that $\pi _{H}(x)=\dot{x}$. The mapping $T_{H}$ is a surjective bounded linear map. Let $\mathcal{J}(G, H)$ be the kernel of $T_H$ in $C_c(G)$:
\[  \mathcal{J}(G, H)= \lbrace k\in C_c(G) \, : \, T_H(k)=0 \rbrace. \]
Let $J^1(G, H)$ be the closure of $\mathcal{J}(G, H)$ in $L^1(G)$. Then 
\[ L^1(G/H)\cong L^1(G)/J^1(G, H).\]
The isomorphism $\cong$ is algebraic and isometric, $L^1(G)/J^1(G, H)$ being provided with 
the quotient norm, and is defined via the extension of the mapping $T_H$, by continuity, to the whole space $L^1(G)$ (\cite[Proposition 3.4.5]{reit}). We denote the extended mapping still by $T_H$ and also write 
\[T_{H}f(\dot{x})=\int_{H}f(x\xi)d\xi \quad (f\in L^1(G) ).\]
$T_H$ maps $L^1(G)$ onto $L^1(G/H)$. $T_H$ is an algebra $*$-homomorphism and the subspace $J^1(G, H)$, i.e. the kernel of $T_H$, is a closed two-sided ideal of $L^1(G)$ (\cite[Theorem 3.5.4 ]{reit}). Note that $H=\{ e \}$ if and only if $J^1(G, H)=0$, where $e $ is the identity of $G$.
\par 
In the following, the $J^1(G, H)$ is the closed ideal in $L^1(G)$ described above, where $G$ is a locally compact group and $H$ is a closed normal subgroup of $G$.
\par 
In view of above discussions and the Johnson theorem, we obtain the following theorem. 
\begin{theorem}\label{jon}
Let $G$ be a locally compact group. For each closed normal subgroup $H$ of $G$ there is a closed ideal $J^1(G, H)$ such that $L^1(G)$ is $J^1(G, H)$-amenable if and only if $G/H$ is amenable as a locally compact group. Moreover, $H=\{ e \}$ ($ e $ is the identity of $G$) if and only if $J^1(G, H)=0$.
\end{theorem}
According to this point that $H=\{ e \}$ if and only if $J^1(G, H)=0$, it follows that Theorem~\ref{jon} is a generalization of the Johnson's theorem in the concept of relative amenability.
\par 
One of the topics of interest is the study of amenability of the second dual of group algebra. In \cite{gha}, the authors showed that $ L^1(G )^{**}$ with the first Arens product is amenable if and only if $G$ is finite. We extend this theorem in the concept of relative amenability. Also we check the relative amenability of the second dual of $L^1(G)$.
\begin{theorem}
 Let $G$ be a locally compact group. For each closed normal subgroup $H$ of $G$ there is a closed ideal $J^1(G, H)$ such that $ L^1(G )^{**}$ with the first Arens product is $J^1(G, H)^{**}$-amenable if and only if $G/H$ is finite.
\end{theorem}
\begin{proof}
The closed $J^1(G, H)$ ideal is obtained as above. We see that $L^1 (G)/J^1(G, H)\cong L^1 (G/H)$ as isometric isomorphism. So by \cite{gha}, $G/H$ is a finite group if and only if $L^1 (G/H)^{**}$ (with the first Arens product) is amenable, if and only if $(L^1 (G)/J^1(G, H))^{**}$ is amenable, if and only if $ L^1(G )^{**}$ is $J^1(G, H)^{**}$-amenable (with the first Arens product). Note that in the above proof, this result is used that by \cite[Page 250]{dals}, $(L^1 (G)/J^1(G, H))^{**}\cong L^1(G )^{**}/J^1(G, H)^{**}$ with the first Arens product.
\end{proof}
Let $G$ be a locally compact group. Let $C_b(G)$ be the Banach algebra of all bounded continuous complex-valued functions on $G$ with the sup norm topology, and $LUC(G)$ denote the closed subspace of all $f\in C_b(G)$ such that the map $x\mapsto l_x f$ from $G$ into $C_b(G)$ is continuous, where $(l_x f)(y)=f(xy)$ and $x,y\in G$, i.e. $f$ is a left uniformly continuous function on $G$. Then $LUC(G)^*$ equipped with the Arens multiplication defined by $\langle nm,f \rangle = \langle n, m_l f \rangle$, $n,m \in LUC(G)^*$, $f\in LUC(G)$, where $m_l f(x)=\langle m, l_x f \rangle$, $x\in G$, is a Banach algebra. Also, the measure algebra $M(G)$ may be identified with a closed subalgebra of $LUC(G)^*$ by the natural embeding $\langle \nu , f \rangle = \int f(x)d\nu (x)$, $f\in LUC(G)$, $\nu\in M(G)$. We denote by $C_0(G)$ the functions in $C_b(G)$ which vanish at infinity. Let $C_0 (G)^\bot = \lbrace m\in LUC(G)^* \, : \, m(f)=0 \, \text{for all } f\in C_0 (G)\rbrace$. From \cite[Lemma 1.1 ]{gha iso}, it follows that $LUC(G)^* =  M(G)\oplus C_0 (G)^\bot$ as direct sum of closed subspaces and $C_0 (G)^\bot$ is a closed ideal in $LUC(G)^*$. Let $E$ be a right identity of $L^1(G)^{**}$ with $E\geq 0$, $\Vert E \Vert=1$. Then $EL^1(G)^{**}$ is a closed sublagebra of $L^1(G)^{**}$. In \cite{gha iso0}, it is shown that $EL^1(G)^{**}\cong LUC(G)^*$ as isomorphism of Banach algebras. So we can consider $C_0 (G)^\bot$ as a closed subspace of $L^1(G)^{**}$. Furthermore, $(I-E)L^1(G)^{**}$ is a closed subspace of $L^1(G)^{**}$ and it is clear that $C_0 (G)^\bot\bigcap (I-E)L^1(G)^{**}=0$. In the following theorem we see that $I=C_0 (G)^\bot\oplus(I-E)L^1(G)^{**}$ is a closed ideal of $L^1(G)^{**}$. Also in this theorem the amenability of $L^1(G)^{**}$ relative to $I$ has been studied.
 \begin{theorem}
  Let $G$ be a locally compact group, and  let $ E$ be a right identity of $L^1(G)^{**}$ with $E\geq 0$, $\Vert E \Vert=1$. Then $I=C_0 (G)^\bot\oplus(I-E)L^1(G)^{**}$ is a closed ideal of $L^1(G)^{**}$, and $ L^1(G )^{**}$ is $I$-amenable if and only if $G$ is discrete and amenable as a group.
 \end{theorem}
\begin{proof}
We have the decomposition 
\[ L^1(G)^{**}= EL^1(G)^{**} \oplus (I-E)L^1(G)^{**}, \]
where $EL^1(G)^{**}$ is a closed subalgebra of $L^1(G)^{**}$ and $(I-E)L^1(G)^{**}$ is a closed ideal of $L^1(G)^{**}$. Also, $EL^1(G)^{**}\cong LUC(G)^* =  M(G)\oplus C_0 (G)^\bot$ where $C_0 (G)^\bot$ is a closed ideal in $LUC(G)^*$ and $M(G)$ is a closed subalgebra of $LUC(G)^*$. So $L^1(G)^{**}\cong M(G)\oplus C_0 (G)^\bot \oplus (I-E)L^1(G)^{**}$. Since $C_0 (G)^\bot$ is a closed ideal in $LUC(G)^*$ and $(I-E)L^1(G)^{**}$ is a closed ideal of $L^1(G)^{**}$, it is easily checked that $I=C_0 (G)^\bot\oplus(I-E)L^1(G)^{**}$ is a closed ideal of $L^1(G)^{**}$. All of the isomorphisms mentioned are Banach algebra isomorphisms. So $L^1(G)^{**}/I \cong M(G)$. In \cite{dals2}, it has been proved that $M(G)$ is amenable as a Banach algebra if and only if $G$ is discrete and amenable as a group. Given this result, the proof of the theorem is complete.
\end{proof}
 Let $G$ be a locally compact group, and $M(G)$ be the measure algebra. The subspace of $M(G)$ consisting of the continuous measures is denoted by $M_c(G)$, so that, for $\nu \in M(G)$, we have $\nu \in M_c(G)$ if and only if $\nu(\lbrace s\rbrace)=0$ ($s\in G$), and the subspace of discrete measures is $M_d(G)$, identified with $l^1(G)$. The subspace $M_c(G)$ is a closed ideal of $M(G)$ and $l^1(G)$ is a closed subalgebra of $M(G)$. In the following theorem we check the amenability of $M(G)$ relative to $M_c(G)$ (see \cite{dals2}). 
\begin{theorem}
Let $G$ be a locally compact group. 
\begin{itemize}
\item[(i)] If $G$ is amenable as a discrete group, then $M(G)$ is amenable relative to $M_c(G)$.
\item[(ii)] If $M(G)$ is amenable relative to $M_c(G)$, then $G$ is an amenable group.
\end{itemize}
\end{theorem}
\begin{proof}
We have $M(G)=l^1(G)\oplus M_c(G) $ as $l^1$-sum of Banach spaces (\cite{dals2}). So $M(G)/M_c(G) \cong l^1(G)$. Now by Johnson's theorem $M(G)$ is $M_c(G)$-amenable if and only if $G$ is amenable as a discrete group. Hence we coclude $(i)$. Also if $G$ is amenable as a discrete group, then $G$ is an amenable locally compact group and we obtain $(ii)$.
\end{proof}
In the following theorem we extend the main result of \cite{dals2} in the concept of relative amenability.
\begin{theorem}
Let $G$ be a locally compact group. For each closed normal subgroup $H$ of $G$ there is a closed ideal $I_H$ such that $M(G)$ is $I_H$-amenable if and only if $G/H$ is discrete and amenable as a group. Moreover, $H=\{ e \}$ if and only if $I_H=0$, where $e $ is the identity of $G$.
\end{theorem}
\begin{proof}
Let $H$ be a closed normal subgroup of $G$. By \cite[Theorem 2.1]{lau1}, there exists a continuous epimorphism $\phi: M(G)\rightarrow M(G/H)$. Then $I_H=Ker(\phi)$ is a closed ideal of $M(G)$. From Remark~\ref{ker}, it follows that $M(G)$ is $I_H$-amenable if and only if $M(G/H)$ is amenable. Now the main result of \cite{dals2} concludes the proof of the theorem. Also by the proof of the \cite[Theorem 2.1]{lau1}, it is obtained that $H=\{ e \}$ if and only if $I_H=0$.
\end{proof}
Let $G$ be a locally compact group, and let $A(G)$ be the Fourier algebra of $G$. Let $E$ be a closed subset of $G$. Define 
\[I(E)=\lbrace u\in A(G) \, : \, u(x)=0 \text{ for every } x\in E \rbrace. \]
$I(E)$ is a closed ideal in $A(G)$. Also $I(\{e\})$ is a non-trivial closed ideal in $A(G)$, where $e $ is the identity of $G$. In the following theorem we consider the $I(H)$-amenability of $A(G)$ for a closed subgroup $H$ of $G$. Especially, we see that always there exists a non-trivial closed ideal $I$ in $A(G)$ such that $A(G)$ is $I$-amenable. 
\begin{theorem}\label{ag}
Let $G$ be a locally compact group, and let $A(G)$ be Fourier algebra of $G$.
\begin{itemize}
\item[(i)] If $H$ is a closed subgroup of $G$, then  $A(G)$ is amenable relative to $I(H)$ if and only if $H$ has an abelian subgroup of finite index.
\item[(ii)] For every closed abelian subgroup $H$ of $G$, the Fourier algebra $A(G)$ is amenable relative to $I(H)$. Especially, $A(G)$ is $I(\{e\})$-amenable, where $e $ is the identity of $G$.
\end{itemize}
\end{theorem}
\begin{proof}
$(i)$ By \cite[Lemma 3.8]{for0}, $A(G)/I(H)$ is is isometrically isomorphism to $A(H)$. On the other hand, by \cite[Colloraly 2.5]{for}, $A(H)$ is amenable if and only if $H$ has an abelian subgroup of finite index. So we conclude the proof of $(i)$.
\par 
$(ii)$ follows immediately from $(i)$.
\end{proof}
\begin{remark}
Let $A$ be any of the Banach algebras $L^1(G)$, $L^1(G)^{**}$ (with first Arens product), $M(G)$ or $A(G)$. In each of the results of this subsection, there are certain conditions under which the Banach algebra $A$ has a non-trivial closed ideal $I$ such that $A$ is $I$-amenable, whenever $A$ is not amenable. For example let $G$ be the \textbf{ax+b} group with underlying manifold $(0,\infty)\times \mathbb{R}$ and group action $(x,y)(z,w)=(az, y+xw)$. The group $G$ does not have any abelian subgroup of finite index and hence $A(G)$ is not amenable. But by Theorem~\ref{ag}-$(ii)$, $A(G)$ is $I(\{e\})$-amenable, when $I(\{e\})$ is a non-trivial closed ideal in $A(G)$. Therefore, these results provide answers to some of the questions in Section 1 for the Banach algebra $A$.
\end{remark}

%
\bibliographystyle{amsplain}

\end{document}